\documentclass{amsart}

\usepackage{amssymb,amsbsy,amsthm,amsmath,graphicx,epsfig}
\usepackage{mathabx,times,hyperref,eucal}

\newtheorem*{thm}{Theorem}
\newtheorem*{prop}{Proposition}

\theoremstyle{remark}

\newtheorem*{rmk}{Remark}

\renewcommand{\rm}[1]{\mathrm{#1}}

\newcommand{\bbT}{\mathbf{T}}
\newcommand{\bbZ}{\mathbf{Z}}

\newcommand{\B}{\mathcal{B}}

\newcommand{\s}{\sigma}

\newcommand{\ol}[1]{\overline{#1}}

\begin{document}

\title{Non-convergence of some non-commuting double ergodic averages}

\author{Tim Austin}
\date{}
\address{Mathematics Institute, Zeeman Building, University of Warwick, Coventry CV4 7AL, United Kingdom}
\email{tim.austin@warwick.ac.uk}

\maketitle

\begin{abstract}
Let $S$ and $T$ be measure-preserving transformations of a probability space $(X,\B,\mu)$. Let $f$ be a bounded measurable functions, and consider the integrals of the corresponding `double' ergodic averages:
\[\frac{1}{n}\sum_{i=0}^{n-1} \int f(S^ix)f(T^ix)\ d\mu(x) \qquad (n\ge 1).\]
We construct examples for which these integrals do not converge as $n\to\infty$.  These include examples in which $S$ and $T$ are rigid, and hence have entropy zero, answering a question of Frantzikinakis and Host.

Our proof begins with a corresponding construction for orthogonal operators on a Hilbert space, and then obtains transformations of a Gaussian measure space from them.
\end{abstract}

Let $S$ and $T$ be measure-preserving transformations of a probability space $(X,\B,\mu)$. Let $f$ and $g$ be bounded measurable functions.  The corresponding `double' ergodic averages are the functions
\begin{equation}\label{eq:aves}
\frac{1}{n}\sum_{i=0}^{n-1} f(S^ix)g(T^ix) \qquad (x \in X,\ n\ge 1).
\end{equation}
These are examples of `multiple' ergodic averages.  They lie beyond the setting of the classical ergodic theorems.  Such averages appear in Furstenberg's ergodic theoretic proof of Szemer\'edi's theorem~\cite{Fur77}, and again in his proof with Katznelson of the multi-dimensional Szemer\'edi theorem~\cite{FurKat78}.

If the transformations $S$ and $T$ commute, then the averages in~\eqref{eq:aves} converge in $L^2$~\cite{ConLes84}.  Convergence of related averages in $L^2$ has been studied in many other works as well, culminating in Walsh's proof of convergence in $L^2$ for a large class of multiple averages whenever the transformations involved generate a nilpotent group~\cite{Walsh12}.

By contrast, even if $S$ and $T$ commute, almost sure convergence remains unknown except under various additional assumptions.  Most famously, in~\cite{Bou88} Bourgain proves the almost sure convergence of~\eqref{eq:aves} for any $S$ when $T = S^2$.  The more recent papers~\cite{HuaShaYe19} and~\cite{DonSun18} contain various other pointwise convergence results, such as for multiple ergodic averages when the transformations commute and generate a distal action.  Those papers also contain a more complete description of the literature.

The pairs of transformations that arise from additive combinatorics via Furstenberg's correspondence principle typically do commute, but it is natural to ask whether convergence could be deduced without this assumption. In general, the answer is No: Furstenberg offers a simple example in~\cite[p40]{Fur81}, and several others have been presented since then.

However, positive results have been found under a variety of extra assumptions.  Most of these fall into one of the following categories:
\begin{itemize}
\item[i.] If one of the transformations, say $S$, has a special structure, then the resulting sequences $(f(S^nx))_n$ for $x \in X$ might be especially simple in ways that help to prove convergence.  For example, convergence holds both in $L^2$ and pointwise almost surely if $S$ has quasi-discrete spectrum.  This follows by the methods in~\cite{LesRitdelaRue03}, which addresses related but different questions around `weak disjointness'.
\item[ii.] If $S$ and $T$ are `sufficiently different', then the sequences are forced to be so unalike that the averages converge because of some unavoidable long-run cancellations.  For example, the averages must converge almost surely to the product of $\int f$ and $\int g$ in case $S$ and $T$ are ergodic and disjoint~\cite{Ber85}, and also when the individual spectral types of $S$ and $T$ are mutually singular~\cite{Ass05}.  In the latter paper Assani also proves a relative version of that result over any factors of $S$ and $T$ for which the relevant convergence is already known.
\item[iii.] In a different direction, one can generalize the averages in~\eqref{eq:aves} by considering expressions of the form
\begin{equation}\label{eq:non-lin}
\frac{1}{n}\sum_{i=0}^{n-1} f(S^{a(i)}x)g(T^{b(i)}x) \qquad (x \in X,\ n\ge 1)
\end{equation}
for positive-integer-valued functions $a$ and $b$ that may be non-linear.  Convergence results are known for various specific choices of such functions.  Some of the strongest are in the papers~\cite{Fra22} and~\cite{FraHos23}, which also include a more complete review of the recent literature.  When comparing the averages~\eqref{eq:aves} with this generalized setting, we refer to them as the `linear' and `non-linear' cases respectively.  The paper~\cite{FraHos23} also extends some of the older counterexamples to the non-linear case.
\end{itemize}

Contrasting with (i) above, in all of the previously published examples of non-convergence in the linear case, both $S$ and $T$ are isomorphic to Bernoulli shifts.  This leaves a gap between the `highly random' transformations that appear in counterexamples and the `highly structured' ones for which the positive results in (i) are known.

The first main result of~\cite{FraHos23} sheds some light into this gap.  It asserts that the averages~\eqref{eq:non-lin} do converge in $L^2$ when $a(n)$ is linear, $b(n)$ is an integer polynomial of degree at least $2$, and $T$ has entropy zero.  In that paper, Frantzikinakis and Host ask explicitly whether a counterexample in the linear case can be found in which either $S$ or $T$ has entropy zero.

From the side of counterexamples, another recent work~\cite{HuaShaYe23} provides some of these with entropy zero when $a(n)$ and $b(n)$ are both integer polynomials of degree at least $5$.

This note provides a new construction of counterexamples to convergence in the linear case of the averages in~\eqref{eq:aves} (and even of the integrals of those averages).  In these examples, $S$ and $T$ are both isomorphic to an arbitrary ergodic Gaussian system.  Referring to $(X,\B,\mu,S)$ for definiteness, this means that it is ergodic and that there is a function $f \in L^2(\mu)$ with the following properties:
\begin{itemize}
\item $\int f = 0$ and $\|f\|_2 = 1$;
\item the process $(f\circ S^n)_n$ has a Gaussian joint distribution and generates the whole sigma-algebra $\B$ modulo $\mu$.
\end{itemize}
These systems are a classical source of examples in ergodic theory: see, for instance,~\cite[Section 8.2 and Chapter 14]{CFS}. They include examples that are rigid and hence have zero entropy.

\begin{thm}
Let $(X,\B,\mu,S)$ be an ergodic Gaussian system and let $f \in L^2(\mu)$ have the properties above.  Then there is another transformation $T$ on the same space, isomorphic to $S$, such that
\begin{itemize}
\item[i.] the process $(f\circ T^n)_n$ is Gaussian with the same distribution as $(f\circ S^n)_n$, and it also generates $\B$ modulo $\mu$, and
\item[ii.] the averages
\[\frac{1}{n}\sum_{i=0}^{n-1}\int f(S^ix)f(T^ix)\ d\mu(x) \qquad (n\ge 1)\]
have limit supremum equal to $1$ and limit infimum equal to $-1$.
\end{itemize}
Examples of $S$ and $T$ arising here can be rigid, and hence have zero entropy.
\end{thm}

\begin{rmk}
After a version of the present note first appeared, Xiangdong Ye communicated to me that the authors of~\cite{HuaShaYe23} can also adapt their construction to the linear case: see~\cite{HuaShaYe24}.
\end{rmk}

The ergodic theory of Gaussian systems is relatively tractable because it depends only on their spectra.  As a result, the heart of our theorem is a corresponding fact about orthogonal operators on real Hilbert spaces.  We establish this first, and then complete the proof of the theorem by constructing a Gaussian measure space.

\begin{prop}
Let $H$ be a real Hilbert space, let $U$ be a weakly mixing orthogonal transformation of it, and let $\xi$ be a star-cyclic unit vector for $U$ (that is, the linear span of $\{U^n\xi:\ n \in \bbZ\}$ is dense in $H$).  Then there is another orthogonal transformation $V$, conjugate to $U$, such that
\begin{itemize}
\item[i$^\prime$.] $\xi$ is also star-cyclic for $V$ and satisfies $\langle U^i\xi,\xi\rangle = \langle V^i\xi,\xi\rangle$ for every $i$, and
\item[ii$^\prime$.] the averages
\begin{equation}\label{eq:ave}
\frac{1}{n}\sum_{i=0}^{n-1} \langle U^i\xi,V^i\xi\rangle \qquad (n\ge 1)
\end{equation}
have limit supremum equal to $1$ and limit infimum equal to $-1$.
\end{itemize}
\end{prop}

\begin{proof}\emph{Step 1.}\quad Since the operator $U$ is weakly mixing, we have
\begin{equation}\label{eq:wk-mix}
\frac{1}{n}\sum_{i=0}^{n-1} \|PU^i\xi\| \to 0 \quad \hbox{as}\ n\to\infty
\end{equation}
whenever $P$ is an orthogonal projection with finite dimensional range.

\vspace{7pt}

\emph{Step 2.}\quad Choose a sequence of positive integers $n_1 < n_2 < \dots$ by the following recursion. Set $n_1 := 1$.  Once $n_1$, \dots, $n_k$ have been chosen, let
\[N_k := \ol{\rm{span}}\{U^i\xi:\ 0 \le i < n_k\},\]
and let $P_k$ be the orthogonal projection onto $N_k$.  Since $N_k$ has finite dimension, the convergence in~\eqref{eq:wk-mix} holds for this projection. Choose $n_{k+1} > n_k$ so large that
\begin{equation}\label{eq:wk-mix-2}
\frac{1}{n_{k+1}}\sum_{i=0}^{n_{k+1}-1}\|P_kU^i\xi\| < \frac{1}{k+1}.
\end{equation}

\vspace{7pt}

\emph{Step 3.}\quad Let $N_k$ and $P_k$ be as defined in Step 2.  Observe that $N_k$ increases with $k$.  Define a pairwise orthogonal sequence of subspaces $M_1$, $M_2$, \dots by
\[M_k := N_k\cap N_{k-1}^\perp,\]
and let $Q_k$ be the orthogonal projection onto $N_k$.  It follows that $N_k$ is the orthogonal sum of $M_1$, $M_2$, \dots, $M_k$, or equivalently
\[P_k = P_{k-1} + Q_k = Q_1 + \cdots + Q_k \quad \forall k\ge 1,\]
setting $P_0 := 0$ for convenience.

Let
\[W := 1 - 2Q_2 - 2Q_4 - \cdots.\]
This is an orthogonal operator: it reverses all the vectors in $M_2$, $M_4$, \dots, and it fixes any vector orthogonal to these subspaces.

Finally, let $V:= WUW$.

\vspace{7pt}

\emph{Step 4.}\quad Since $\xi$ lies in $M_1$, it is orthogonal to all of $M_2$, $M_4$, \dots, and so
\begin{equation}\label{eq:V-and-U-on-x}
V^i\xi = WU^i\xi \quad \forall i\in \bbZ.
\end{equation}
Therefore, since $\xi$ is star-cyclic for $U$ and $W$ is an isomorphism, $\xi$ is also star-cyclic for $V$.  In addition,~\eqref{eq:V-and-U-on-x} implies that
\[\langle V^i\xi,\xi\rangle = \langle WU^i\xi,\xi\rangle = \langle U^i\xi,W\xi\rangle = \langle U^i\xi,\xi\rangle.\]
This proves property (i$^\prime$).

\vspace{7pt}

\emph{Step 5.}\quad It remains to prove (ii$^\prime$).  Intuitively,~\eqref{eq:wk-mix-2} and~\eqref{eq:V-and-U-on-x} combine into the following behaviour:
\begin{quote}
For $i$ up to the time $n_1$, $V^ix$ agrees with $U^ix$.  Beyond that, for `most' $i$ up to the much later time $n_2$, `most of' the vector $V^ix$ is the reflection of `most of' $U^ix$ through the origin.  Beyond that, for `most' $i$ up to the still much later time $n_3$, `most of' the vector $V^ix$ agrees with $U^ix$ again.  These behaviours continue to alternate after that.
\end{quote}

To explain this precisely, let $i < n_k$.  Then $U^i\xi$ lies in $N_k$, and so
\begin{equation}\label{eq:full-decomp}
U^i\xi = P_kU^i\xi = P_{k-1}U^i\xi + Q_k U^i\xi.
\end{equation}
As a result, when $n = n_k$, the average in~\eqref{eq:ave} is equal to
\begin{equation}\label{eq:decompose}
\frac{1}{n_k}\sum_{i=0}^{n_k-1} \langle P_{k-1}U^i\xi,P_{k-1}V^i\xi\rangle + \frac{1}{n_k}\sum_{i=0}^{n_k-1} \langle Q_kU^i\xi,Q_kV^i\xi\rangle.
\end{equation}
The first of these two sums is controlled by~\eqref{eq:wk-mix-2}:
\[\Big|\frac{1}{n_k}\sum_{i=0}^{n_k-1} \langle P_{k-1}U^i\xi,P_{k-1}V^i\xi\rangle\Big| \le \frac{1}{n_k}\sum_{i=0}^{n_k-1} \|P_{k-1}U^i\xi\| < \frac{1}{k}.\]
To estimate the second sum, observe that
\[Q_kW = \left\{\begin{array}{ll}Q_k &\quad \hbox{if $k$ is odd}\\ -Q_k &\quad \hbox{if $k$ is even.}\end{array}\right.\]
Combined with~\eqref{eq:V-and-U-on-x}, this gives
\[Q_kV^i\xi = Q_kW U^i W\xi = (-1)^{k-1}Q_k U^i \xi.\]
Now we insert these ingredients back into~\eqref{eq:decompose} and apply~\eqref{eq:full-decomp} and then~\eqref{eq:wk-mix-2}:
\begin{align*}
\frac{1}{n_k}\sum_{i=0}^{n_k-1} \langle U^i\xi,V^i\xi\rangle &= O(1/k) + \frac{(-1)^{k-1}}{n_k}\sum_{i=0}^{n_k-1} \|Q_kU^i\xi\|^2\\
&= O(1/k) + \frac{(-1)^{k-1}}{n_k}\sum_{i=0}^{n_k-1} (1 - \|P_{k-1}U^i\xi\|^2)\\
&= O(1/k) + (-1)^{k-1}(1 - O(1/k)).
\end{align*}
This approaches $1$ (resp. $-1$) as $k$ grows along odd (resp. even) values.
\end{proof}

\begin{proof}[Proof of the theorem]
Let $H$ be the closed span of $\{f\circ S^n:\ n \in \bbZ\}$ in $L^2(\mu)$, let $\xi := f$, and let $U$ be the restriction of the Koopman operator of $S$ to $H$.  Since $S$ is an ergodic Gaussian dynamical system, it is actually weakly mixing~\cite[Theorem 14.2.1]{CFS}, and hence so is the restriction $U$.  We may therefore apply the proposition to find another orthogonal transformation $V$ conjugate to $U$ which has properties (i$^\prime$) and (ii$^\prime$).

Because the process $(f\circ S^n)_n$ is Gaussian and generates $\B$, the measure space $(X,\B,\mu)$ is canonically isomorphic to the Gaussian measure space constructed from $H$: see, for instance,~\cite[Appendices B--D]{KecGAEGA}.  In this picture, the whole of $L^2_0(\mu)$ is identified with the symmetric Fock space over $H$, and $H$ itself appears as a linear subspace called the `first chaos'.  Any orthogonal transformation of the first chaos extends canonically to a measure-preserving transformation of $(X,\B,\mu)$~\cite[Appendix E]{KecGAEGA}.  In the case of $U$, this extension is the transformation $S$ that we started with.  The extension of $V$ is a new transformation $T$, and any orthogonal operator that conjugates $U$ and $V$ extends to an isomorphism between $S$ and $T$.

All transformations obtained from orthogonal operators in this way preserve the first chaos $H$.  Therefore, since $f$ lies in $H$, the images $f\circ T^n$ stay in $H$, and hence define a Gaussian process.  Now property (i) follows from property (i$^\prime$) because $f$ is star-cyclic under $V$ as well as under $U$, and because
\[\int f(T^ix)f(x)\,d\mu(x) = \langle V^i\xi,\xi\rangle = \langle U^i\xi,\xi\rangle = \int f(S^ix)f(x)\,d\mu(x) \quad \forall i \in \bbZ.\]
Similarly, since $f\circ S^n$ and $f\circ T^n$ both lie in $H$ for all $n$, the integrals in assertion (ii) are simply the inner products in assertion (ii$^\prime$) of the proposition, and so (ii$^\prime$) implies (ii).

Finally, the distribution of an ergodic stationary centred Gaussian process is uniquely specified by its spectral measure $\s$, and this can be any atomless finite Borel measure on $\bbT$ that is symmetric under complex conjugation~\cite[Section 8.2]{CFS}.  Having chosen $\s$ for our process $(f\circ S^n)_n$, the maximal spectral type of the whole transformation $S$ on $L^2_0(\mu)$ is given by
\[\s + \frac{1}{2!}\s\ast \s + \frac{1}{3!}\s\ast\s\ast\s + \cdots\]
(see, for instance,~\cite[Theorem 14.3.1]{CFS}).  In particular, if $\s$ is atomless but has a rigidity sequence, such as the measures discussed in~\cite[Sections 7.14--18]{NadSTDS}, then $S$ has the same rigidity sequence, and then so does $T$ because they are isomorphic.  In this case their maximal spectral type has no absolutely continuous component, and so they also have entropy zero.
\end{proof}

\begin{rmk}
The function $f$ obtained above is not bounded.  However, given that those averages of inner products fail to converge, the same must be true for sufficiently high truncations of $f$, so bounded examples are also available.
\end{rmk}

\begin{rmk}
Let $a(1) < a(2) < \dots $ be positive integers and let $b(i) := a(i)$.  The theorem and proposition above can be extended to the resulting non-linear averages (as in~\eqref{eq:non-lin}) whenever the analog of~\eqref{eq:wk-mix} holds for the averages along the subsequence of times $(a(i))_i$, as pointed out to me by Sohail Farhangi.  If $U$ is strongly mixing, then this condition holds for any subsequence of times, and some strongly mixing examples given by our construction still have entropy zero.  On the other hand, if $a(i)$ is an increasing polynomial function of $i$, then the analog of~\eqref{eq:wk-mix} holds under the same hypotheses as in the linear case.  This is because the weak mixing of $S$ (or $U$) implies the required analog of~\eqref{eq:wk-mix} by applying the norm ergodic theorem along polynomial sequences to the weakly mixing operator $U\otimes U$.
\end{rmk}

The known examples of non-convergence of the averages~\eqref{eq:aves} provide rather diverse measure-preserving transformations in the role of $S$ (or $T$).  Many further questions about such transformations could be pursued. For example, it would be interesting to know whether \emph{any} weakly mixing transformation can appear as $S$; whether the ranks of $S$ and $T$ could help to control the averages in~\eqref{eq:aves} (Gaussian examples always have infinite rank, and so do those in~\cite{HuaShaYe24}); or which two-step distal transformations can appear.

In~\cite{FraHos23}, Frantzikinakis and Host also pose various questions about recurrence involving non-commuting transformations.  I do not think the methods of the present note come close to answering these.

\subsection*{Acknowledgements}

I am grateful to Mariusz Lema\'nczyk for bringing recent work in this area to my attention, and to Nikos Frantzikinakis, Xiangdong Ye, el Houcein el Abdalaoui, Sohail Farhangi and an anonymous referee for some helpful remarks about earlier versions of this note.

For the purpose of open access, the author has applied a Creative Commons Attribution (CC-BY) licence to any Author Accepted Manuscript version arising from this submission.




\begin{thebibliography}{{remark}02}

\bibitem{Ass05} I. Assani, ``Pointwise convergence of nonconventional averages'', \emph{Colloq. Math.} 102 (2005), no. 2, 245--262

\bibitem{Ber85} D. Berend, ``Joint ergodicity and mixing'', \emph{J. Analyse Math.} 45 (1985), 255--284

\bibitem{Bou88} J. Bourgain, ``Double recurrence and almost sure convergence'', \emph{J. Reine Angew. Math.} 404 (1990), 140–161

\bibitem{ConLes84} J.-P. Conze and E. Lesigne, ``Th\'eor\`emes ergodiques pour des mesures diagonales'', \emph{Bull. Soc. Math. France} 112 (1984), no. 2, 143--175

\bibitem{CFS} I.P. Cornfeld, S.V. Fomin and Ya. G. Sina\u\i, \emph{Ergodic Theory}, Springer-Verlag, New York, 1982

\bibitem{DonSun18} S. Donoso and W. Sun, ``Pointwise convergence of some multiple ergodic averages'', \emph{Adv. Math.} 330 (2018), 946–996

\bibitem{Fra22} N. Frantzikinakis, ``Furstenberg systems of {H}ardy field sequences and applications'', \emph{J. Anal. Math.} 147 (2022), no. 1, 333--372

\bibitem{FraHos23} N. Frantzikinakis and B. Host, ``Multiple recurrence and convergence without commutativity'', \emph{J. Lond. Math. Soc. (2)} 107 (2023), no. 5, 1635--1659

\bibitem{Fur77} H. Furstenberg, ``Ergodic behaviour of diagonal measures and a theorem of {S}zemer\'edi on arithmetic progressions'', \emph{J. d'Analyse Math.} 31 (1977), 204--256

\bibitem{Fur81} H. Furstenberg, \emph{Recurrence in Ergodic Theory and Combinatorial Number Theory}, Princeton University Press, Princeton, 1981

\bibitem{FurKat78} H. Furstenberg and Y. Katznelson, ``An ergodic {S}zemer\'edi theorem for commuting transformations'', \emph{J. d'Analyse Math.} 34 (1978), 275--291

\bibitem{HuaShaYe19} W. Huang, S. Shao and X. Ye, ``Pointwise convergence of multiple ergodic averages and strictly ergodic models'', \emph{J. Anal. Math} 139 (2019), no.1, 265--305

\bibitem{HuaShaYe23} W. Huang, S. Shao and X. Ye, ``A counterexample on polynomial multiple convergence without commutativity'', eprint arXiv:2301.12409

\bibitem{HuaShaYe24} W. Huang, S. Shao and X. Ye, ``A counterexample on multiple convergence without commutativity'', eprint arXiv:2407.10728

\bibitem{KecGAEGA} A.S. Kechris, \emph{Global Aspects of Ergodic Group Actions}, Mathematical Surveys and Monographs vol. 160, American Mathematical Society, 2010


\bibitem{LesRitdelaRue03}  E. Lesigne, B. Rittaud and T. de la Rue, ``Weak disjointness of measure-preserving dynamical systems'', \emph{Ergodic Theory Dynam. Systems} 23 (2003), no. 4, 1173--1198

\bibitem{NadSTDS} M. Nadkarni, \emph{Spectral Theory of Dynamical Systems}, Springer, Singapore, 2020

\bibitem{Walsh12} M.N. Walsh, ``Norm convergence of nilpotent ergodic averages'', \emph{Ann. of Math. (2)} 175 (2012), no. 3, 1667--1688


%
%
%
%
%
%
%
%
%
%
%
%
%

\end{thebibliography}
\end{document}